\UseRawInputEncoding
\documentclass[12pt]{amsart}
\usepackage{graphics}
\usepackage{latexsym}
\usepackage{amsmath}
\usepackage{amssymb,amsthm,amsfonts}
\usepackage{lscape}
\usepackage{amscd}
\usepackage{syntonly}
\title[Dolbeault type cohomology groups ]
{Dolbeault type cohomology groups of Infinitesimal deformations and Their Applications}

\author[Jiezhu Lin]{Jiezhu Lin}
\author[Xuanming Ye]{Xuanming Ye}

\address{School of Mathematics and Information Science, Guangzhou University, Key Laboratory of Mathematics and Interdisciplinary Sciences of Guangdong Higher Education Institutes, No 230, Waihuan Road West, Guangzhou, 510006, P.R.China}
\address{School of Mathematics and Information Science, Guangzhou University, Key Laboratory of Mathematics and Interdisciplinary Sciences of Guangdong Higher Education Institutes, No 230, Waihuan Road West, Guangzhou, 510006, P.R.China}
\email{jlin@gzhu.edu.cn}
\email{yexm3@gzhu.edu.cn}

\begin{document}
\theoremstyle{plain}
\newtheorem{thm}{Theorem}[section]
\newtheorem{theorem}[thm]{Theorem}
\newtheorem{lemma}[thm]{Lemma}
\newtheorem{corollary}[thm]{Corollary}
\newtheorem{proposition}[thm]{Proposition}
\newtheorem{addendum}[thm]{Addendum}
\newtheorem{variant}[thm]{Variant}
\theoremstyle{definition}
\newtheorem{construction}[thm]{Construction}
\newtheorem{notations}[thm]{Notations}
\newtheorem{question}[thm]{Question}
\newtheorem{problem}[thm]{Problem}
\newtheorem{remark}[thm]{Remark}
\newtheorem{remarks}[thm]{Remarks}
\newtheorem{definition}[thm]{Definition}
\newtheorem{claim}[thm]{Claim}
\newtheorem{assumption}[thm]{Assumption}
\newtheorem{assumptions}[thm]{Assumptions}
\newtheorem{properties}[thm]{Properties}
\newtheorem{example}[thm]{Example}
\newtheorem{conjecture}[thm]{Conjecture}
\numberwithin{equation}{section}

\newcommand{\sA}{{\mathcal A}}
\newcommand{\sB}{{\mathcal B}}
\newcommand{\sC}{{\mathcal C}}
\newcommand{\sD}{{\mathcal D}}
\newcommand{\sE}{{\mathcal E}}
\newcommand{\sF}{{\mathcal F}}
\newcommand{\sG}{{\mathcal G}}
\newcommand{\sH}{{\mathcal H}}
\newcommand{\sI}{{\mathcal I}}
\newcommand{\sJ}{{\mathcal J}}
\newcommand{\sK}{{\mathcal K}}
\newcommand{\sL}{{\mathcal L}}
\newcommand{\sM}{{\mathcal M}}
\newcommand{\sN}{{\mathcal N}}
\newcommand{\sO}{{\mathcal O}}
\newcommand{\sP}{{\mathcal P}}
\newcommand{\sQ}{{\mathcal Q}}
\newcommand{\sR}{{\mathcal R}}
\newcommand{\sS}{{\mathcal S}}
\newcommand{\sT}{{\mathcal T}}
\newcommand{\sU}{{\mathcal U}}
\newcommand{\sV}{{\mathcal V}}
\newcommand{\sW}{{\mathcal W}}
\newcommand{\sX}{{\mathcal X}}
\newcommand{\sY}{{\mathcal Y}}
\newcommand{\sZ}{{\mathcal Z}}
\newcommand{\A}{{\mathbb A}}
\newcommand{\B}{{\mathbb B}}
\newcommand{\C}{{\mathbb C}}
\newcommand{\D}{{\mathbb D}}
\newcommand{\E}{{\mathbb E}}
\newcommand{\F}{{\mathbb F}}
\newcommand{\G}{{\mathbb G}}
\newcommand{\HH}{{\mathbb H}}
\newcommand{\I}{{\mathbb I}}
\newcommand{\J}{{\mathbb J}}
\renewcommand{\L}{{\mathbb L}}
\newcommand{\M}{{\mathbb M}}
\newcommand{\N}{{\mathbb N}}
\renewcommand{\P}{{\mathbb P}}
\newcommand{\Q}{{\mathbb Q}}
\newcommand{\R}{{\mathbb R}}
\newcommand{\SSS}{{\mathbb S}}
\newcommand{\T}{{\mathbb T}}
\newcommand{\U}{{\mathbb U}}
\newcommand{\V}{{\mathbb V}}
\newcommand{\W}{{\mathbb W}}
\newcommand{\X}{{\mathbb X}}
\newcommand{\Y}{{\mathbb Y}}
\newcommand{\Z}{{\mathbb Z}}
\newcommand{\id}{{\rm id}}
\newcommand{\rank}{{\rm rank}}
\newcommand{\END}{{\mathbb E}{\rm nd}}
\newcommand{\End}{{\rm End}}
\newcommand{\Eis}{{\rm Eis}}
\newcommand{\Hg}{{\rm Hg}}
\newcommand{\tr}{{\rm tr}}
\newcommand{\Sl}{{\rm Sl}}
\newcommand{\Gl}{{\rm Gl}}
\newcommand{\Gr}{{\rm Gr}}
\newcommand{\Cor}{{\rm Cor}}
\newcommand{\dcechg}{\delta\!\!\!\check{\delta}}
\newcommand{\del}{\partial}
\newcommand{\delbar}{\overline{\del}}
\newcommand{\tempo}{\mathbf{t}}
\newcommand\SO{{\rm SO}}
\newcommand\SU{{\rm SU}}
\newcommand\fra{{\mathfrak a}} 
\newcommand\frg{{\mathfrak g}}
\newcommand\frh{{\mathfrak h}}
\newcommand\frk{{\mathfrak K}}
\newcommand\frn{{\mathfrak n}}
\newcommand\frs{{\mathfrak s}}
\newcommand\frt{{\mathfrak t}}
\newcommand\gc{\frg_\mathbb{C}}
\newcommand\Real{{\mathfrak R}{\frak e}\,} 
\newcommand\Imag{{\mathfrak I}{\frak m}\,}
\newcommand\nilm{\Gamma\backslash G}
\newcommand\Mod{{\cal M}}
\newcommand\db{{\bar{\partial}}}
\newcommand\zzz{{\!\!\!}}
\newcommand{\la}{\langle}
\newcommand{\sla}{\slash\!\!\!}
\newcommand{\opiccolo}[1]{\mathrm{o}\left(\left|#1\right|\right)}
\newcommand{\ra}{\rangle}
\newcommand\sqi{{\sqrt{-1\,}}}

\maketitle


\begin{abstract}\ \rm In this paper, we establish a kind of Dolbeault type cohomology groups for the purpose of
studying the varying of complex structure invariants in infinitesimal deformations of any order. We give a concrete description
of the higher order Kodaria-Spencer maps by using these cohomology groups. We reformulate the obstruction formulas within
the framework of these cohomology groups and give a new proof for the formulas.
\vskip10pt
\noindent
{\bf MSC~Classification (2020):} 32G05, 32L10, 55N30, 32G99
\end{abstract}

\section{Introduction}
\renewcommand{\theequation}
{1.\arabic{equation}} \setcounter{equation}{-1}
\indent
In the classic deformation theory of complex structures, there are two different approaches.
One of them is a more algebraic style, using algebraic methods. The other is the style of analysis, using more complex analysis tools. The former usually focuses on studying the $n$th order neighborhoods of complex structure deformations. In this case, when
we study the cohomology of the locally free sheaves over $n$th order neighborhoods, besides the $\check{C}$ech cohomology groups there are seldom  ways to do the calculations. In this paper we establish a kind of Dolbeault type cohomology groups to overcome this difficulty. Firstly, we carefully selected a structure sheaf as the basis for our approach. More precisely, let $X\xrightarrow{i}\mathcal{X} \xrightarrow{\pi}B$ be a small deformation of $X$, we take a sheaf $\mathcal{C}^{\omega\infty}_\mathcal{X}$ of, roughly speaking, $\mathcal{C}^{\infty}$-functions on $\mathcal{X}$ which are real analytic in direction of $B$ as the structure sheaf. Any holomorphic vector bundle $ pr:E \rightarrow \mathcal{X}$ over $\mathcal{X}$ can be considered as  a locally free sheaf of $\sO_\sX^{\omega}$-module. In this paper,  we construct a $\Gamma$-acyclic resolution \ref{resolution2} of its restriction to $X^{\omega}_{\infty}$, from which we obstain a Dolbeault type cohomology groups to compute the cohomology groups of this sheaf
$$
H^p\left( X_{\infty}^{\omega},\mathcal{E} _{\mathcal{X} |X_{\infty}^{\omega}} \right) \cong H_{\bar{\partial}_{\sX/B}}^{p}\left( \mathcal{E} _{\mathcal{X} |X_{\infty}^{\omega}} \right).
$$
We also discuss the case of its restriction to $X^{\omega}_n$ and a $\Gamma$-acyclic resolution \ref{resolution1} of it and similarly we also get a Dolbeault type cohomology groups to compute the cohomology groups
$$
H^p\left( X_{n}^{\omega},\mathcal{E} _{\mathcal{X} |X_{n}^{\omega}} \right) \cong H_{\bar{\partial}_{X_n/B_n}}^{p}\left( \mathcal{E} _{\mathcal{X} |X_{n}^{\omega}} \right)
.$$
\indent After the construction of the Dolbeault type cohomology groups, we obtain a concrete description of Kodaira-Spencer map of $n$th order by using the cohomology class in the cohomolgy groups $H_{\bar{\partial}_{\sX/B}}^{p}( \mathcal{T} _{\mathcal{X}/B|X_{\infty}^{\omega}}) $  and $H_{\bar{\partial}_{X_n/B_n}}^{p}( \mathcal{T} _{\mathcal{X}/B|X_{n}^{\omega}}) $. These cohomology classes are described by using the Beltrami differential $\varphi(t)$ in theorem \ref{main2}
\begin{align*}
\kappa : \Gamma \left( B,\mathcal{T} _{B^{\omega}} \right)& \longrightarrow   H^1\left(  X_{\infty}^{\omega},\mathcal{T} _{X_{\infty}^{\omega}/B^{\omega}} \right)&
\\
       u & \longmapsto  [\rho d_u\left( \varphi \left( t \right) \right)]&
\end{align*}
and theorem \ref{main3}
\begin{align*}
\kappa_n:
\mathcal{T} _{B_{n}^{\omega}|B_{n-1}^{\omega}} & \longrightarrow H^1\left( X_{n-1}^{\omega},\mathcal{T} _{X_{n}^{\omega}/B^{\omega}_{n-1}} \right) &
\\
       u & \longmapsto  [r_{n-1}(\rho d_{\tilde{u}}\left( \varphi \left( t \right) \right))].&
\end{align*}
\indent As an application of the above discussion, we reformulate the obstruction formula in theorem 3.5  in \cite{YEH} and theorem 3.4  in \cite{YET},
that is theorem \ref{omain} and respectively theorem \ref{omain2} in last section. And  new proofs of these theorems are given by using the analytic tools at the end of this paper.

\section{Dolbeault type cohomology groups of infinitesimal deformations}\label{section2}
 \renewcommand{\theequation}
{2.\arabic{equation}} \setcounter{equation}{-1}
\subsection{Some Notations. } \ \ \\
\indent Here we make some definitions in order to construct the Dolbeault type cohomology groups of infinitesimal deformations.
Let $\pi:\mathcal{X}\longrightarrow B$ be a deformation of
$\pi^{-1}(0)=X$, where $X$ is a compact complex manifold of dimension $m$ and $i:X \longrightarrow \mathcal{X}$ be the inclusion map. Let $ pr:E \longrightarrow \mathcal{X}$ be a holomorphic vector bundle over $\mathcal{X}$ and there exists a diffeomorphism $E\cong E_0 \times \mathcal{X}$ where $E_0= E_{|X}$. Let $\mathcal{E}_{\mathcal{X}}$ be the locally free sheaf of $\mathcal{O}_{\mathcal{X}}$-module associated to $E$. For
every integer $n\geq 0$, denote by
$B_{n}=Spec\,\mathcal{O}_{B,0}/m_{0}^{n+1}$ the $n$th order
infinitesimal neighborhood of the closed point $0\in B$ of the
base $B$. Let $X_{n}\subset \mathcal{X}$ be the complex space over
$B_{n}$. Let $\pi_{n}:X_{n}\longrightarrow B_{n}$ be the $n$th order
deformation of $X$.  \\
\indent Let $\sC^{\omega}_B$ be the sheaf of
$\C$-valued real analytic functions on $B$ and $B^{\omega}$ be the complex space which is the topology space $B$ equipped with the structure sheaf  $\sC^{\omega}_B$. 
Let $m^{\omega}_{0}$ be the maximal ideal of $\sC^{\omega}_{B,0}$.
 Let ${\sC}_{B_n}^{\omega}= \sC_{B,0}^{\omega}/ ({m}_{0}^{\omega})^{n+1}$£¬and $B^{\omega}_n$ be the complex space
which is the topology space $B$ being equipped with the structure sheaf ${\sC}_{B_n}^{\omega}$.
Denote $\sO_\sX^{\omega}= \pi^{*}\sC^{\omega}_B$ and
$\mathcal{M}_{0}^{\omega}= \pi^{*}m^{\omega}_{0} $. Let ${\sO}_{X_n}^{\omega}= \sO_{\sX}^{\omega}/ ({\mathcal{M}}_{0}^{\omega})^{n+1}$£¬and $X^{\omega}_n$(resp. $X^{\omega}_{\infty}$) be the complex space
which is the topology space $X$ being equipped with the structure sheaf ${\sO}_{X_n}^{\omega}$(resp. $i^{-1}\sO_{\sX}^{\omega}$). Denote $\mathcal{E}_{\mathcal{X}|X^{\omega}_n}$ the sheaf $i^{-1} \mathcal{E}_{\mathcal{X}} \otimes_{\sO_{X_n}} {\sO}_{X_n}^{\omega}$ over $X^{\omega}_n$ and  $\mathcal{E}_{\mathcal{X}|X^{\omega}_{\infty}}$ the sheaf $i^{-1} \mathcal{E}_{\mathcal{X}} \otimes_{i^{-1}\sO_{\sX}} i^{-1}{\sO}_{\sX}^{\omega}$ over $X^{\omega}_{\infty}$.
In the following we want to construct Dolbeault type cohomology groups which are isomorphic to $H^{\cdot}(X^{\omega}_n,\mathcal{E}_{\mathcal{X}|X_n^{\omega}})$ and $H^{\cdot}(X^{\omega}_{\infty},\mathcal{E}_{\mathcal{X}|X_{\infty}^{\omega}})$. \ \ \\
\indent  In order to define Dolbeault type cohomolgy groups, one has to consider the sheaves of $\mathcal{C}^{\infty}$-module over the fibers and therefore the object with a mixture of $\mathcal{C}^{\infty}$  and real analytic data  is  a very suitable for us to study. Let $\mathcal{C}^{\omega\infty}_\mathcal{X}$ be the sheaf of $\mathcal{C}^{\infty}$-functions on $\mathcal{X}$ which are real analytic in direction of $B$. More precisely, according to a complex version of the theorem of Ehresmann, there exists a diffeomorphism(\cite{Voi02I}, Proposition 9.5)  $$ T=(T_0,\pi):\mathcal{X} \cong X_0\times B$$ having the following property: that the fibres of $T_0$ are complex submanifolds of $\mathcal{X}$.
Without ambiguity, we can talk about the $\mathcal{C}^{\infty}$-functions on $X_0 \times B$ which are real analytic in direction of $B$, and denote the sheaf of these functions by $\mathcal{C}^{\omega\infty}_{X_0\times B}$ and $\mathcal{C}^{\omega\infty}_\mathcal{X}$ is defined to be $T^{-1}\mathcal{C}^{\omega\infty}_{X_0\times B}$ .
\begin{remark}
An equivalent definition of the sheaf $\mathcal{C}^{\omega\infty}_\mathcal{X}$ is that it is the sheaf of $\mathcal{C}^{\infty}$ functions whose restriction on every fibres of $T_0$ are  real analytic functions. Note that both definitions depend on the choice of the diffeomorphism $T$, so a more precise notation for this sheaf should be $\mathcal{C}^{\omega\infty,T}_\mathcal{X}$. However, the diffeomorphism $T$ will be fixed in rest of our discussion, so for simplicity, we used $\mathcal{C}^{\omega\infty}_\mathcal{X}$ to denote it.
\end{remark}
 Then we make some definitions of sheaves over $\mathcal{X}$.\\
\textbullet \quad Let $\mathcal{M}_0^{\omega\infty}=\pi^{*}m_0\otimes_{\mathcal{O}_\mathcal{X}}\sC^{\omega\infty}_{\sX}$ , $\sE^{\omega\infty}_{\sX}=\sE_{\sX}\otimes_{\sO_{\sX}}\sC^{\omega\infty}_{\sX}$, $\sA^{\cdot,0}_{\sX/B}=\Omega^{\cdot}_{\sX/B}\otimes_{\sO_{\sX}}\sC^{\omega\infty}_{\sX}$, $\sA^{\cdot,0}_{\sX}=\Omega^{\cdot}_{\sX}\otimes_{\sO_{\sX}}\sC^{\omega\infty}_{\sX}$,  $\sA^{0,\cdot}_{\sX/B}=\bar{\Omega}^{\cdot}_{\sX/B}\otimes_{\bar{\sO}_{\sX}}\sC^{\omega\infty}_{\sX}$,
$\sA^{0,\cdot}_{\sX}=\bar{\Omega}^{\cdot}_{\sX}\otimes_{\bar{\sO}_{\sX}}\sC^{\omega\infty}_{\sX}$,
 $\sA^{\cdot}_{\sX/B}(\sE_{\sX})=\sA^{0,\cdot}_{\sX/B}\otimes_{\sC^{\omega\infty}_{\sX}}\sE^{\omega\infty}_{\sX}$ and  $\sA^{\cdot}_{\sX}(\sE_{\sX})=\sA^{0,\cdot}_{\sX}\otimes_{\sC^{\omega\infty}_{\sX}}\sE^{\omega\infty}_{\sX}$. \\
\indent Definitions of sheaves over $X_n$ is also needed. \\
 \textbullet \quad Let $\sC^{\omega\infty}_{X_n}$ be the quotient sheaf of $i^{-1}\sC^{\omega\infty}_{\sX}$ by $(i^{-1}\sM^{\omega\infty}_{0})^n$,
 $\mathcal{E}^{\omega\infty}_{X_n}=i^{-1}\sE^{\omega\infty}_{\sX} \otimes_{i^{-1}\sC^{\omega\infty}_{\sX}}\sC^{\omega\infty}_{X_n}$, $\sA^{\cdot,0}_{X_n/B_n}= \Omega^{\cdot}_{X_n/B_n} \otimes_{\sO_{X_n}}\sC^{\omega\infty}_{X_n}$, $\sA^{0,\cdot}_{X_n/B_n}= \bar{\Omega}^{\cdot}_{X_n/B_n} \otimes_{\bar{\sO}_{X_n}}\sC^{\omega\infty}_{X_n}$(the map $\sO_{X_n} \longrightarrow \sC^{\omega\infty}_{X_n}$ is just  the quotient of $i^{-1}\sO_{\sX} \longrightarrow i^{-1}\sC^{\omega\infty}_{\sX}$),
 and $\sA^{\cdot}_{X_n/B_n}(\sE^{\omega\infty}_{X_n})= \sA^{0,\cdot}_{X_n/B_n}\otimes_{\sC^{\omega\infty}_{X_n}}\sE^{\omega\infty}_{X_n}$.\\
 \indent At the end we need to construct the Dolbeault type Operators in the following. Note that the holomorphic structure of $E$ induce Dolbeault operators
 $$ \bar{\partial}_{\sX}: \sA^{p}_{\sX}(\sE_{\sX}) \longrightarrow \sA^{p+1}_{\sX}(\sE_{\sX}). $$
 From the short exact sequence of sheaves over $\sX$,
 $$ 0 \longrightarrow (\pi^{*}\bar{\Omega}_B \otimes_{\bar{\sO}_{\sX}} \sC^{\omega\infty}_{\sX})\wedge_{\sC^{\omega\infty}_{\sX}} \sA^{0,p}_{\sX} \longrightarrow \sA^{0,p+1}_{\sX} \longrightarrow \sA^{0,p+1}_{\sX/B} \longrightarrow 0.$$
  for all $p\geq 0$,
  we get operators
  $$ \bar{\partial}_{\sX/B}: \sA^{p}_{\sX/B}(\sE_{\sX}) \longrightarrow \sA^{p+1}_{\sX/B}(\sE_{\sX}). $$ by simply taking the quotient of $ \bar{\partial}_{\sX}$. Then easy computation shows that the operator $\bar{\partial}_{\sX/B}$ is $\pi^{-1}\sC^{\omega}_B$-linear. So we have
   $$ \bar{\partial}_{\sX/B}( \sA^{p}_{\sX/B}(\sE_{\sX}) \otimes_{\sC^{\omega\infty}_{\sX}} (\mathcal{M}_0^{\omega\infty})^n)\subset  \sA^{p+1}_{\sX/B}(\sE_{\sX}) \otimes_{\sC^{\omega\infty}_{\sX}} (\mathcal{M}_0^{\omega\infty})^n.$$
Which allow us to construct Dolbeault type opeartors
 $$ \bar{\partial}_{X_n/B_n}:\sA^{p}_{X_n/B_n}(\sE^{\omega\infty}_{X_n}) \longrightarrow \sA^{p+1}_{X_n/B_n}(\sE^{\omega\infty}_{X_n})$$
 by simply taking the quotient of $\bar{\partial}_{\sX/B}$.
 \subsection{Definitions and Isomorphisms. } \ \ \\
 In order to finish the construction, we also need the following two lemmas.
 \begin{lemma} \label{soft lemma}
 We have $H^{k}(X_n,\sA^{p}_{X_n/B_n}(\sE^{\omega\infty}_{X_n}))=0, \forall k>0$ .
 \end{lemma}
 \begin{proof} Note that we have the follow exact sequences of sheaves over $X$,
 $$ 0 \xrightarrow{} i^{-1}\sA^{p}_{\sX/B}(\sE_{\sX}) \otimes_{i^{-1}\sC^{\omega\infty}_{\sX}} (i^{-1}\mathcal{M}_0^{\omega\infty})^n \xrightarrow{} i^{-1}\sA^{p}_{\sX/B}(\sE_{\sX}) \xrightarrow{r_n} \sA^{p}_{X_n/B_n}(\sE^{\omega\infty}_{X_n}) \longrightarrow 0.$$
 Since $\sC^{\infty}_{X}$ is a sub-sheaf of $i^{-1}\sC^{\omega\infty}_{\sX}$,  $i^{-1}\sC^{\omega\infty}_{\sX}$ is a sheaf with the property of partition of unity. So as sheaves of $i^{-1}\sC^{\omega\infty}_{\sX}$-module we conclude that
 $i^{-1}\sA^{p}_{\sX/B}(\sE_{\sX}) \otimes_{i^{-1}\sC^{\omega\infty}_{\sX}} (i^{-1}\mathcal{M}_0^{\omega\infty})^n $ and $i^{-1}\sA^{p}_{\sX/B}(\sE_{\sX})$ are fine. Which means
 $H^{k}(X_n,i^{-1}\sA^{p}_{\sX/B}(\sE_{\sX}) \otimes_{i^{-1}\sC^{\omega\infty}_{\sX}} (i^{-1}\mathcal{M}_0^{\omega\infty})^n)=0, \forall k>0$  and $H^{k}(X_n,i^{-1}\sA^{p}_{\sX/B}(\sE_{\sX}))=0, \forall k>0$ . The proof is finished if we take the long exact sequences associated to the above short exact sequences.
 \end{proof}
 \indent The ordinary $\bar{\partial}$-Poincar\'{e} lemma that every local $\bar{\partial}$-closed form is $\bar{\partial}$-exact assures us that the Dolbeault cohomology groups are locally trivial. Analogously, a fundamental fact in our case is the following lemma.
\begin{lemma} ($\bar{\partial}_{X_n/B_n}$-Poincar\'{e} lemma) \label{Poincare lemma}
For any point $x \in X$, there exists an open neighborhood $U$ of $x$, such that any $\bar{\partial}_{X_n/B_n}$-closed sections of $\sA^{\cdot}_{X_n/B_n}(\sE^{\omega\infty}_{X_n})(U)$
is $\bar{\partial}_{X_n/B_n}$-exact.
\end{lemma}
\begin{proof} Let $U$ be a small open neighborhood of $x$, note that we have the following short exact sequences of complexes of sections on $U$.
\begin{align*}
0\longrightarrow i^{-1}\mathcal{A} _{\mathcal{X} /B}^{0}(\mathcal{E} _{\mathcal{X}}) & \otimes _{i^{-1}\mathcal{C} _{\sX}^{\omega \infty}}&(i^{-1}\mathcal{M} _{0}^{\omega \infty})^n (U) \longrightarrow i^{-1}    & \mathcal{A}  _{\mathcal{X} /B}^{0}(\mathcal{E} _{\mathcal{X}})(U)&\longrightarrow  &\mathcal{A} _{X_n/B_n}^{0}(\mathcal{E} _{X_n}^{\omega \infty})(U)&\longrightarrow 0
\\
& \downarrow _{_{\bar{\partial}_{\mathcal{X} /B}}}  &           &   \downarrow _{_{\bar{\partial}_{\mathcal{X} /B}}} &       &\downarrow _{_{\bar{\partial}_{X_n/B_n}}}&
\\
0\longrightarrow i^{-1}\mathcal{A} _{\mathcal{X} /B}^{1}(\mathcal{E} _{\mathcal{X}}) & \otimes _{i^{-1}\mathcal{C} _{\sX}^{\omega \infty}}&(i^{-1}\mathcal{M} _{0}^{\omega \infty})^n(U)\longrightarrow i^{-1}&\mathcal{A} _{\mathcal{X} /B}^{1}(\mathcal{E} _{\mathcal{X}})(U)&\longrightarrow &\mathcal{A} _{X_n/B_n}^{1}(\mathcal{E} _{X_n}^{\omega \infty})(U)&\longrightarrow 0
\\
       &      \downarrow _{_{\bar{\partial}_{\mathcal{X} /B}}}  &              & \downarrow _{_{\bar{\partial}_{\mathcal{X} /B}}}    &       & \downarrow _{_{\bar{\partial}_{X_n/B_n}}}&
\\
          &  ...   &                &    ...     &          &   ...&
\\
       &  \downarrow _{_{\bar{\partial}_{\mathcal{X} /B}}}    &          & \downarrow _{_{\bar{\partial}_{\mathcal{X} /B}}}   &      &  \downarrow _{_{\bar{\partial}_{X_n/B_n}}}&
\\
0\longrightarrow i^{-1}\mathcal{A} _{\mathcal{X} /B}^{m}(\mathcal{E} _{\mathcal{X}})& \otimes _{i^{-1}\mathcal{C} _{\sX}^{\omega \infty}}&(i^{-1}\mathcal{M} _{0}^{\omega \infty})^n(U)\longrightarrow i^{-1}&\mathcal{A} _{\mathcal{X} /B}^{m}(\mathcal{E} _{\mathcal{X}})(U)&\longrightarrow &\mathcal{A} _{X_n/B_n}^{m}(\mathcal{E} _{X_n}^{\omega \infty})(U)&\longrightarrow 0
\\
       &  \downarrow _{_{\bar{\partial}_{\mathcal{X} /B}}}    &          & \downarrow _{_{\bar{\partial}_{\mathcal{X} /B}}}   &      &  \downarrow _{_{\bar{\partial}_{X_n/B_n}}}&
\\
          &  0  &                &    0     &          &   0&
\end{align*}
Then the exactness of complex in the middle column comes from the  $\bar{\partial}$-Poincar\'{e} lemma directly (one may modified $U$ to get the exactness if it is necessary). Together with the $\bar{\partial}$-Poincar\'{e} lemma and the $\pi^{-1}\sC^{\omega}_B$-linearity of $\bar{\partial}_{\sX/B}$ we get the exactness of the complex in the left column.
Then if we take the long sequence of cohomologies associate to the above short exact sequences, we obtain the exactness of the complex in the right column.
\end{proof}
By lemma \ref{soft lemma} and lemma \ref{Poincare lemma}, one can tell that the following complex of sheaves over $X_n^{\omega}$
\begin{align} \label{resolution1}
\begin{split}
&\mathcal{E} _{\mathcal{X} |X_{n}^{\omega}}\longrightarrow \mathcal{A} _{X_n/B_n}^{0}\left( \mathcal{E} _{X_n}^{\omega \infty} \right) \xrightarrow{\bar{\partial}_{X_n/B_n}}\mathcal{A} _{X_n/B_n}^{1}\left( \mathcal{E} _{X_n}^{\omega \infty} \right) \xrightarrow{\bar{\partial}_{X_n/B_n}}\mathcal{A} _{X_n/B_n}^{2}\left( \mathcal{E} _{X_n}^{\omega \infty} \right) \\
& \xrightarrow{\bar{\partial}_{X_n/B_n}}...\xrightarrow{\bar{\partial}_{X_n/B_n}}\mathcal{A} _{X_n/B_n}^{m}\left( \mathcal{E} _{X_n}^{\omega \infty} \right)
\end{split}
\end{align}
is a $\Gamma$-acyclic resolution of $\mathcal{E} _{\mathcal{X} |X_{n}^{\omega}}$ . Besides the following complex
\begin{align}  \label{resolution2}
\begin{split}
&\mathcal{E} _{\mathcal{X} |X_{\infty}^{\omega}}\longrightarrow i^{-1}\mathcal{A} _{\mathcal{X} /B}^{0}\left( \mathcal{E} _{\mathcal{X}} \right) \xrightarrow{\bar{\partial}_{\mathcal{X} /B}}i^{-1}\mathcal{A} _{\mathcal{X} /B}^{1}\left( \mathcal{E} _{\mathcal{X}} \right) \xrightarrow{\bar{\partial}_{\mathcal{X} /B}}i^{-1}\mathcal{A} _{\mathcal{X} /B}^{2}\left( \mathcal{E} _{\mathcal{X}} \right)
\\
&\xrightarrow{\bar{\partial}_{\mathcal{X} /B}}...\xrightarrow{\bar{\partial}_{\mathcal{X} /B}}i^{-1}\mathcal{A} _{\mathcal{X} /B}^{m}\left( \mathcal{E} _{\mathcal{X}} \right)
\end{split}
\end{align}

is also a $\Gamma$-acyclic resolution. Thus we obtain the following isomorphisms.
\begin{theorem} \label{Dol}
The Dolbeault type cohomology groups associated to $\mathcal{E}_{\mathcal{X}|X_{n}^{\omega}}$ over $X_{n}^{\omega}$ and $X_{\infty}^{\omega}$ are
$$
H_{\bar{\partial}_{X_n/B_n}}^{p}\left( \mathcal{E} _{\mathcal{X} |X_{n}^{\omega}} \right) :=\frac{\mathrm{Ker}\left( \bar{\partial}_{X_n/B_n}:\Gamma \left( X,\mathcal{A} _{X_n/B_n}^{p}\left( \mathcal{E} _{X_n}^{\omega \infty} \right) \right) \longrightarrow \Gamma \left( X,\mathcal{A} _{X_n/B_n}^{p+1}\left( \mathcal{E} _{X_n}^{\omega \infty} \right) \right) \right)}{\mathrm{Im}\left( \bar{\partial}_{X_n/B_n}:\Gamma \left( X,\mathcal{A} _{X_n/B_n}^{p-1}\left( \mathcal{E} _{X_n}^{\omega \infty} \right) \right) \longrightarrow \Gamma \left( X,\mathcal{A} _{X_n/B_n}^{p}\left( \mathcal{E} _{X_n}^{\omega \infty} \right) \right) \right)}
$$
and respectively
$$
H_{\bar{\partial}_{\sX/B}}^{p}\left( \mathcal{E} _{\mathcal{X} |X_{\infty}^{\omega}} \right) :=\frac{\mathrm{Ker}\left( \bar{\partial}_{\sX/B}:\Gamma \left( X,i^{-1}\mathcal{A} _{\sX/B}^{p}\left( \mathcal{E} _{\mathcal{X}} \right) \right) \longrightarrow \Gamma \left( X,i^{-1}\mathcal{A} _{\sX/B}^{p+1}\left( \mathcal{E} _{\mathcal{X}} \right) \right) \right)}{\mathrm{Im}\left( \bar{\partial}_{\sX/B}:\Gamma \left( X,i^{-1}\mathcal{A} _{\sX/B}^{p-1}\left( \mathcal{E} _{\mathcal{X}} \right) \right) \longrightarrow \Gamma \left( X,i^{-1}\mathcal{A} _{\sX/B}^{p}\left( \mathcal{E} _{\mathcal{X}} \right) \right) \right)}
$$
where $p\geq 0$ and $\Gamma \left( X,\mathcal{A} _{X_n/B_n}^{-1}\left( \mathcal{E} _{X_n}^{\omega \infty} \right) \right) =0$, $\Gamma \left( X,i^{-1}\mathcal{A} _{\sX/B}^{-1}\left( \mathcal{E} _{\mathcal{X}} \right) \right) =0$. \\
Then
$$
H^p\left( X_{n}^{\omega},\mathcal{E} _{\mathcal{X} |X_{n}^{\omega}} \right) \cong H_{\bar{\partial}_{X_n/B_n}}^{p}\left( \mathcal{E} _{\mathcal{X} |X_{n}^{\omega}} \right)
$$
and
$$
H^p\left( X_{\infty}^{\omega},\mathcal{E} _{\mathcal{X} |X_{\infty}^{\omega}} \right) \cong H_{\bar{\partial}_{\sX/B}}^{p}\left( \mathcal{E} _{\mathcal{X} |X_{\infty}^{\omega}} \right).
$$
\end{theorem}
\begin{remark} \label{dulremark}
Similar construction can be created when we consider an anti-holomorphic bundle $\bar{\sE}_{\sX}$ over $\sX$. In this case, $\bar{\sE}^{\omega\infty}_{\sX}=\bar{\sE}_{\sX}\otimes_{\bar{\sO}_{\sX}}\sC^{\omega\infty}_{\sX}$, $\sA^{\cdot}_{\sX/B}(\bar{\sE}_{\sX})=\sA^{\cdot,0}_{\sX/B}\otimes_{\sC^{\omega\infty}_{\sX}}\bar{\sE}^{\omega\infty}_{\sX}$, $\sA^{\cdot}_{X_n/B_n}(\bar{\sE}^{\omega\infty}_{X_n})= \sA^{\cdot,0}_{X_n/B_n}\otimes_{\sC^{\omega\infty}_{X_n}}\bar{\sE}^{\omega\infty}_{X_n}$ and the counterpart of operators are $\partial_{\sX/B}$ on
$\sA^{\cdot}_{\sX/B}(\bar{\sE}_{\sX})$ and $\partial_{X_n/B_n}$ on $\sA^{\cdot}_{X_n/B_n}(\bar{\sE}^{\omega\infty}_{X_n})$.
If we take $\Omega _{\mathcal{X} /B}^{p}$ as a holomorphic sheaf and take $\bar{\Omega} _{\mathcal{X} /B}^{q}$ an anti-holomorphic sheaf over $\sX$, one may see that
$\mathcal{A} _{\mathcal{X} /B}^{q}\left( \Omega _{\mathcal{X} /B}^{p} \right) =\mathcal{A} _{\mathcal{X} /B}^{0,q}\otimes _{\mathcal{C} _{\mathcal{X}}^{\omega \infty}}\left( \Omega _{\mathcal{X} /B}^{\omega \infty} \right) ^p
$ and $\mathcal{A} _{\mathcal{X} /B}^{p}\left( \bar{\Omega}_{\mathcal{X} /B}^{q} \right) =\mathcal{A} _{\mathcal{X} /B}^{p,0}\otimes _{\mathcal{C} _{\mathcal{X}}^{\omega \infty}}\left( \bar{\Omega}_{\mathcal{X} /B}^{\omega \infty} \right) ^q
$. So as sheaves of $\mathcal{C} _{\mathcal{X}}^{\omega \infty}$-module, $\mathcal{A} _{\mathcal{X} /B}^{p}\left( \bar{\Omega}_{\mathcal{X} /B}^{q} \right) =\mathcal{A} _{\mathcal{X} /B}^{q}\left( \Omega _{\mathcal{X} /B}^{p} \right)
$ and we have both of  the operators $\partial_{\sX/B}$ and $\bar{\partial}_{\sX/B}$ operating on the sections of this sheaf. After taking the quotient of it to the finite order $n$, we
get $\mathcal{A} _{X_n/B_n}^{q}(\Omega _{X_n}^{p,\omega \infty})=\mathcal{A} _{X_n/B_n}^{p}(\bar{\Omega}_{X_n}^{q,\omega \infty})
$ and therefore  both of  the operators $\partial_{X_n/B_n}$ and $\bar{\partial}_{X_n/B_n}$ operating on the sections of these sheaf.

\end{remark}


\section{concrete description of Kodaira-Spencer map of n-th order}\label{section3}
\renewcommand{\theequation}
{3.\arabic{equation}} \setcounter{equation}{-1}
Recall that the classical Kodiara-Spencer Map is given by the connecting homomorphism of
the complex (\cite{Voi02I}, (9.1))
$$
0\longrightarrow T_X\longrightarrow T_{\mathcal{X} |X}\longrightarrow \pi^* T_{B|X}  \longrightarrow 0.
$$
The above definition is actually the first order Kodaria-Spencer map and gives us a classification of the first order deformation of $X$ induced by the deformation $\sX$. In this section, we will generalize this map to any order $n$ of infinitesimal deformations(including $n=\infty$) and use the cohomology groups defined in theorem \ref{Dol} to describe the
Kodaira-Spencer map concretely.\\
\indent First, we consider the case $n=\infty$. In this case, we consider the short exact sequences

\begin{align} \label{ex1}
0\longrightarrow \mathcal{T} _{{X} _{\infty}^{\omega}/B^{\omega}}\longrightarrow \mathcal{T} _{{X} _{\infty}^{\omega}}\longrightarrow i^{-1}\left( \pi ^{-1}\mathcal{T} _{B^{\omega}}\otimes _{\pi ^{-1}\mathcal{C} _{B}^{\omega}}\mathcal{O} _{\mathcal{X}}^{\omega} \right) \longrightarrow 0,
\end{align}
and the Kodaria-Spencer map is given by the connecting homomorphism
$$
\delta : H^0\left( X_{\infty}^{\omega},i^{-1}\left( \pi ^{-1}\mathcal{T} _{B^{\omega}}\otimes _{\pi ^{-1}\mathcal{C} _{B}^{\omega}}\mathcal{O} _{\mathcal{X}}^{\omega} \right) \right) =\Gamma \left( B,\mathcal{T} _{B^{\omega}} \right) \longrightarrow H^1\left( X_{\infty}^{\omega},\mathcal{T} _{X_{\infty}^{\omega}/B^{\omega}} \right) .
$$
Note that $
\mathcal{T} _{X_{\infty}^{\omega}/B^{\omega}}\,\,
$ is exactly $\mathcal{T} _{\mathcal{X} /B|X_{\infty}^{\omega}}\,\,$, and therefore we may use the cohomology class in the cohomology group
$H_{\bar{\partial}_{\sX/B}}^{p}( \mathcal{T} _{\mathcal{X}/B|X_{\infty}^{\omega}}) $ to give a description of the above map. \\
\indent Some preparations is needed for the description, first, let us recall that associated to the given diffeomorphism $T$ in section\ref{section 2}, there is a so called "Beltrami differential" $
\varphi \left( t \right)
$ which is a vector $(0,1)$-form on $X$ with the parameter $t$. A precise construction of $
\varphi \left( t \right)
$ can be found in \cite{Voi02I}(Proposition 9.7)(see also \cite{MK71}, $\S$ 4.1). Two fundamental facts about Beltrami differential are that it satisfies the Maurer-Cartan equation
$$
\bar{\partial}\varphi \left( t \right) =\frac{1}{2}\left[ \varphi \left( t \right) ,\varphi \left( t \right) \right],
$$
 $\varphi(0)=0$ and $\varphi(t)$ varies analytically with the parameter $t$, where "$[\cdot,\cdot]$" is the Lie black of the vector $(0,\cdot)$-forms on $X$.

\indent To finish the preparation, we still need a new useful operator defined in paper \cite{XW} $\S 2.1.4$
$$ \rho: A^{0,q}(X,E) \longrightarrow A^{0,q}(X,E_{t}), $$
where $E$ is a holomorphic tensor bundle on $X$ and $E_t$ the corresponding holomorphic tensor bundle on $X_t$.
This operator provides us a better understanding of the Dolbeault Operator $\bar{\partial}_t$ on $X_t$ given by theorem 2.9 in \cite{XW}.
\begin{theorem}\cite{XW} {\label{XWT}}
Let $\bar{\partial}$ and $\bar{\partial}_t$  be the Dolbeault Operator on $X$ and $X_t$, respectively. Then we have
$$
\rho ^{-1}\bar{\partial}_t\rho =\bar{\partial}-\left< \varphi \left( t \right) |, \right. \,\,   on\,\,  A^{0.\cdot}\left( X,E \right)
$$
In particular, for $\sigma \in A^{0,q}\left( X,E \right)$ , $\rho \sigma A^{0,q}\left( X,E_t \right)$ is $\bar{\partial}_t$ -closed if and only if
$$ \bar{\partial}-\left< \varphi \left( t \right) |\left. \sigma \right> =0. \right. $$
\end{theorem}
By this theorem, we get the following lemma immediately,
\begin{lemma} \label{closed lemma}
For any section $u \in \Gamma \left( B,\mathcal{T} _{B^{\omega}} \right) $, $\rho d_u(\varphi(t))$ is $\bar{\partial}_{\sX/B}$-closed and thus $[\rho d_u(\varphi(t))]$
is a cohomology class in   $H^1\left( X_{\infty}^{\omega},\mathcal{T} _{X_{\infty}^{\omega}/B^{\omega}} \right).$
\end{lemma}
\begin{proof}
By the Maurer-Cartan equation of Beltrami differential, we get
$$
\bar{\partial}\varphi \left( t \right) -\frac{1}{2}\left[ \varphi \left( t \right) ,\varphi \left( t \right) \right]=0,
$$
Then use $u$ to operator on both side of the equation above, we get
$$
\bar{\partial}d_u(\varphi \left( t \right)) -\frac{1}{2}\left[ d_u(\varphi \left( t \right)) ,\varphi \left( t \right) \right]-\frac{1}{2}\left[ (\varphi \left( t \right) ,d_u(\varphi \left( t \right)) \right]=0.
$$
Then since $\varphi(t)$ and $d_u(\varphi(t))$ are both (0,1)-forms,
$$
\bar{\partial}d_u(\varphi \left( t \right)) -\left[ (\varphi \left( t \right) ,d_u(\varphi \left( t \right)) \right]=0.
$$
And thus we finished the proof by using the theorem \ref{XWT}.
\end{proof}
\indent The following result shows that by using the cohomology class appears in the above lemma we have a concrete description of the connection morphism of the
exact sequence \ref{ex1}.
\begin{theorem} \label{main2}
The map given by
\begin{align*}
\kappa : \Gamma \left( B,\mathcal{T} _{B^{\omega}} \right)& \longrightarrow   H^1\left(  X_{\infty}^{\omega},\mathcal{T} _{X_{\infty}^{\omega}/B^{\omega}} \right)&
\\
       u & \longmapsto  [\rho d_u\left( \varphi \left( t \right) \right)]&
\end{align*}
is the connection morphism
\begin{equation} \label{KSM}
\delta :  \Gamma \left( B,\mathcal{T} _{B^{\omega}} \right) \longrightarrow H^1\left( X_{\infty}^{\omega},\mathcal{T} _{X_{\infty}^{\omega}/B^{\omega}} \right) .
\end{equation}
of the exact sequence \ref{ex1}. We call this map $\kappa$ the Kodaria-Spencer map of order $\infty$.
\end{theorem}
\begin{proof}
The proof is given by local computations. So we need to choose a local coordinates system. At first, we choose a coordinate neighborhood $U$ of a point $0\in B$ and local coordinates
$\{t^1,...,t^s\}$ of $U$ center $0$. Let $\pi^{-1}(U)$ be covered by open sets $\{\sU_j\}_{j\in \Lambda}$: $
\cup_{j\in \Lambda}^{}{\mathcal{U} _j}=\pi ^{-1}\left( U \right), \sU_j=U_j \times U
$ where $\{U_J\}$ is an open cover of $X$(Note that we only focus on the small deformation, we may assume $\sX=\pi^{-1}(U)$ and $B=U$).
Since $\pi$ is a holomorphic map, we can choose local coordinates of $\sU_j$ as $
\left( \zeta _{j}^{1},\zeta _{j}^{2},...,\zeta _{j}^{m},t^1,t^2,...,t^s \right)
$. (That is $t^1,t^2,...,t^s$ is chosen as part of the local coordinates of $\sU_j$ for all $j$.) If $\sU_j \cap \sU_k \neq \Phi $, between the two local coordinates we have
the following relation:
$$
\zeta _{j}^{\alpha}=f_{jk}^{\alpha}\left( \zeta _{j}^{1},...,\zeta _{j}^{m},t^1,...,t^s \right) , \alpha=1,2,...,m.
$$
By direct computation (\cite{YU}, $\S$ 1.1.2), one will find that the map
\begin{align*}
\Gamma \left( B,\mathcal{T} _{B^{\omega}} \right)& \longrightarrow   \check{H}^1\left(  X_{\infty}^{\omega},\mathcal{T} _{X_{\infty}^{\omega}/B^{\omega}} \right)&
\\
       u & \longmapsto \left\{ \sum_{\alpha =1}^n{d_u\left( f_{jk}^{\alpha} \right) \left( \zeta ^1,...,\zeta ^m,t^1,...,t^s \right)}\frac{\partial}{\partial \zeta _{j}^{\alpha}} \right\}
&
\end{align*}
(where $\left\{ \sum_{\alpha =1}^n{d_u\left( f_{jk}^{\alpha} \right) \left( \zeta ^1,...,\zeta ^m,t^1,...,t^s \right)}\frac{\partial}{\partial \zeta _{j}^{\alpha}} \right\}$ is a $\check{C}$ech one-cocycle with coefficients in holomorphic tangent vector fields) is the connection morphism $\delta$.

 Let $z_j=
\left( z_{j}^{1},z_{j}^{2},...,z_{j}^{m} \right)
$ be the local coordinates of $U_j$, then we may take $ ( z_{j}^{1},z_{j}^{2},...,z_{j}^{m},t^{1},t^{2},...,t^{s})$ as a differentiable coordinates on $\sU_j$.
One may consider $\zeta^{\alpha}_j$ as a function $\zeta^{\alpha}_j(z_j,t)$  of  $ ( z_{j}^{1},z_{j}^{2},...,z_{j}^{m},t^{1},t^{2},...,t^{s})$, we have
$$
d_u\left( \zeta _{j}^{\alpha} \right) =\sum_{\beta}{\frac{\partial f_{jk}^{\alpha}}{\partial \zeta _{k}^{\beta}}d_u\left( \zeta _{k}^{\beta} \right) +d_u\left( f_{jk}^{\alpha} \right)}
$$
Thus,
\begin{align*}
\sum_{\alpha =1}^m{d_u\left( f_{jk}^{\alpha} \right) }\frac{\partial}{\partial \zeta _{j}^{\alpha}} &=
\sum_{\alpha}{d_u\left( \zeta _{j}^{\alpha} \right)}\left( \frac{\partial}{\partial \zeta _{j}^{\alpha}} \right) -\sum_{\alpha}{d_u\left( \zeta _{k}^{\alpha} \right) \left( \frac{\partial \zeta _{j}^{\alpha}}{\partial \zeta _{k}^{\alpha}} \right)}\left( \frac{\partial}{\partial \zeta _{j}^{\alpha}} \right)&\\
&=\sum_{\alpha}{d_u\left( \zeta _{j}^{\alpha} \right)}\left( \frac{\partial}{\partial \zeta _{j}^{\alpha}} \right) -\sum_{\alpha}{d_u\left( \zeta _{k}^{\alpha} \right)}\left( \frac{\partial}{\partial \zeta _{k}^{\alpha}} \right).&
\end{align*}
In order to finish the proof, we need to show that on each $\sU_j$,
\begin{align} \label{top}
\bar{\partial}_{\mathcal{X} /B}\left( \sum_{\alpha}{d_u\left( \zeta _{j}^{\alpha} \right)}\left( \frac{\partial}{\partial \zeta _{j}^{\alpha}} \right) \right) =\rho d_u\left( \varphi \left( t \right) \right).
\end{align}
In fact, on each $\sU_j$, the local form of $\varphi(t)$ is given by (\cite{MK71}, $\S$ 4.1)
$$
\varphi \left( t \right) =\sum_{\alpha ,\lambda}{A_{j\alpha}^{\lambda}\bar{\partial}\zeta _{j}^{\alpha}\left( z_j,t \right) \frac{\partial}{\partial z_{j}^{\lambda}}},
$$
where $$
A_{j\alpha}^{\lambda}=\left( \frac{\partial \zeta _{j}^{\alpha}}{\partial z^{\lambda}_j} \right) ^{-1}.
$$

Being operated by $d_u$ and we get
\begin{align*}
d_u\left( \varphi \left( t \right) \right) &=\sum_{\alpha ,\lambda}{d_u\left( A_{j\alpha}^{\lambda} \right) \bar{\partial}\zeta _{j}^{\alpha}\left( z_j,t \right) \frac{\partial}{\partial z_{j}^{\lambda}}}+\sum_{\alpha ,\lambda}{A_{j\alpha}^{\lambda}d_u\left( \bar{\partial}\zeta _{j}^{\alpha}\left( z_j,t \right) \right) \frac{\partial}{\partial z_{j}^{\lambda}}}&
\\
&=-\sum_{\alpha ,\lambda,\beta,\nu}{A_{j\nu}^{\lambda}d_u\left( \frac{\partial \zeta _{j}^{\nu}}{\partial z_{j}^{\beta}} \right) A_{j\alpha}^{\beta}\bar{\partial}\zeta _{j}^{\alpha}\left( z_j,t \right) \frac{\partial}{\partial z_{j}^{\lambda}}}+\sum_{\alpha ,\lambda}{A_{j\alpha}^{\lambda}\bar{\partial}\left( d_u\left( \zeta _{j}^{\alpha}\left( z_j,t \right) \right) \right) \frac{\partial}{\partial z_{j}^{\lambda}}}&
\end{align*}
Then
\begin{align*}
\rho d_u\left( \varphi \left( t \right) \right) =&-\sum_{\alpha ,\lambda ,\beta ,\nu ,\mu}{A_{j\nu}^{\lambda}d_u\left( \frac{\partial \zeta _{j}^{\nu}}{\partial z_{j}^{\beta}} \right) A_{j\alpha}^{\beta}\bar{\partial}\zeta _{j}^{\alpha}\left( z_j,t \right) \frac{\partial \zeta _{j}^{\mu}}{\partial z_{j}^{\lambda}}\frac{\partial}{\partial \zeta _{j}^{\mu}}}\\
&+\sum_{\alpha ,\lambda ,\mu}{A_{j\alpha}^{\lambda}\bar{\partial}\left( d_u\left( \zeta _{j}^{\alpha}\left( z_j,t \right) \right) \right) \frac{\partial \zeta _{j}^{\mu}}{\partial z_{j}^{\lambda}}\frac{\partial}{\partial \zeta _{j}^{\mu}}}&
\\
=&-\sum_{\alpha ,\beta ,\nu}{d_u\left( \frac{\partial \zeta _{j}^{\nu}}{\partial z_{j}^{\beta}} \right) A_{j\alpha}^{\beta}\bar{\partial}\zeta _{j}^{\alpha}\left( z_j,t \right) \frac{\partial}{\partial \zeta _{j}^{\nu}}}+\sum_{\alpha}{\bar{\partial}\left( d_u\left( \zeta _{j}^{\alpha}\left( z_j,t \right) \right) \right) \frac{\partial}{\partial \zeta _{j}^{\alpha}}}.&
\end{align*}
On the other hand,
\begin{align*}
\bar{\partial}_{\mathcal{X} /B}\left( \sum_{\alpha}{d_u\left( \zeta _{j}^{\alpha}\left( z_j,t \right) \right)}\left( \frac{\partial}{\partial \zeta _{j}^{\alpha}} \right) \right) =&\bar{\partial}_{\mathcal{X} /B}\left( \sum_{\alpha}{d_u\left( \zeta _{j}^{\alpha}\left( z_j,t \right) \right)} \right) \frac{\partial}{\partial \zeta _{j}^{\alpha}}
\\
=&\left( \bar{\partial}-\sum_{\lambda ,\beta}{A_{j\beta}^{\lambda}\bar{\partial}\zeta _{j}^{\beta}\left( z_j,t \right) \frac{\partial}{\partial z_{j}^{\lambda}}} \right) \left( \sum_{\alpha}{d_u\left( \zeta _{j}^{\alpha}\left( z_j,t \right) \right)} \right) \frac{\partial}{\partial \zeta _{j}^{\alpha}}
\\
=&\sum_{\alpha}{\bar{\partial}\left( d_u\left( \zeta _{j}^{\alpha}\left( z_j,t \right) \right) \right)}\frac{\partial}{\partial \zeta _{j}^{\alpha}}\\
&-\sum_{\lambda ,\beta ,\alpha}{A_{j\beta}^{\lambda}\bar{\partial}\zeta _{j}^{\beta}\left( z_j,t \right) d_u\left( \frac{\partial \zeta _{j}^{\alpha}}{\partial z_{j}^{\lambda}} \right)}\frac{\partial}{\partial \zeta _{j}^{\alpha}}.
\end{align*}
So we get \ref{top} and finish the proof.
\end{proof}
\begin{remark}
The usual way to define the Kodaria-Spencer map is given by the connection morphism of the short exact sequence(\cite{YU} $\S$1.1.2)
$$
0\longrightarrow \mathcal{T} _{\mathcal{X} /B}\longrightarrow \mathcal{T} _{\mathcal{X}}\longrightarrow \pi ^*\mathcal{T} _B\longrightarrow 0,
$$
and use the $\check{C}$ech cohomology class to describe the map. It is obviously that $\mathcal{T}_B$ is a sub sheaf of $\mathcal{T}_B^{\omega}$.
 If we take the restriction of our map \ref{KSM} to the subset  $\Gamma \left( B,\mathcal{T} _{B} \right)$, we will come back to the usual
Kodaria-Spencer map. Therefore the map \ref{KSM} can be considered as a kind of generalization of the usual one which can be used to study the
real-analytic variety on the local moduli space of complex structures. Furthermore the description of the map by using the Dolbeault type cohomology groups
in $\S$\ref{section2} provides a way to use analytic tools to study the cohomology groups related to infinitesimal deformations for us.

\end{remark}
\indent For the case when $n$ is finite,  we consider the short exact sequences (a generalization of \cite{Ye4} (1.24))
\begin{align} \label{ex2}
0 \longrightarrow \mathcal{T} _{X_{n-1}^{\omega}/B_{n-1}^{\omega}}\longrightarrow \mathcal{T} _{X_{n}^{\omega}|X_{n-1}^{\omega}}\longrightarrow \pi ^{-1}\mathcal{T} _{B_{n}^{\omega}|B_{n-1}^{\omega}}\otimes _{\pi ^{-1}\mathcal{C} _{B_{n-1}}^{\omega}}\mathcal{O} _{X_{n-1}}^{\omega}\longrightarrow 0.
\end{align}
and the Kodaria-Spencer map is given by the connecting homomorphism
$$
\delta : H^0\left( X_{n-1}^{\omega},\pi ^{-1}\mathcal{T} _{B_{n}^{\omega}|B_{n-1}^{\omega}}\otimes _{\pi ^{-1}\mathcal{C} _{B_n}^{\omega}}\mathcal{O} _{X_{n-1}}^{\omega} \right)  =\mathcal{T} _{B_{n}^{\omega}|B_{n-1}^{\omega}} \longrightarrow H^1\left( X_{n-1}^{\omega},\mathcal{T} _{X_{n}^{\omega}/B^{\omega}_{n-1}} \right) .
$$
Here we define the sheaf  $\mathcal{T} _{X_{n}^{\omega}|X_{n-1}^{\omega}}$ as the sheaf of derivations of $\sO^{\omega}_{X_n}$ with values in $\sO^{\omega}_{X_{n-1}}$ and $\mathcal{T} _{B_{n}^{\omega}|B_{n-1}^{\omega}}$ as the sheaf of derivations of $\sC^{\omega}_{B_n}$ with values in $\sC^{\omega}_{B_{n-1}}$.
Note that $
\mathcal{T} _{X_{n}^{\omega}/B^{\omega}}\,\,
$ is exactly $\mathcal{T} _{\mathcal{X} /B|X_{n}^{\omega}}\,\,$, and therefore we may use the cohomology class in the cohomology group
$H_{\bar{\partial}_{X_n/B_n}}^{p}( \mathcal{T} _{\mathcal{X}/B|X_{n}^{\omega}}) $ to give a description of the above map. 
Note that the  Kodaria-Spencer map of finite order is just the restriction of the Kodaria-Spencer map over $X^{\omega}_{\infty}$ which has been given in theorem \ref{main2}. So we have the following theorem for the cases of finite order.\\
\begin{theorem} \label{main3}
The map given by
\begin{align*}
\kappa_n:
\mathcal{T} _{B_{n}^{\omega}|B_{n-1}^{\omega}} & \longrightarrow H^1\left( X_{n-1}^{\omega},\mathcal{T} _{X_{n}^{\omega}/B^{\omega}_{n-1}} \right) &
\\
       u & \longmapsto  [r_{n-1}(\rho d_{\tilde{u}}\left( \varphi \left( t \right) \right))]&
\end{align*}
where $r_{n-1}$ is the quotient map form  $i^{-1}\sA^{p}_{\sX/B}(\sE_{\sX})$ to  $\sA^{p}_{X_n/B_n}(\sE^{\omega\infty}_{X_n})$ and $\tilde{u}$ is a section on $\mathcal{T} _{B^{\omega}}$
such that $r_{n-1}\tilde{u}=u$
is the connection morphism
\begin{align} \label{KSM2}
\delta :  \mathcal{T} _{B_{n}^{\omega}|B_{n-1}^{\omega}} \longrightarrow H^1\left( X_{n-1}^{\omega},\mathcal{T} _{X_{n}^{\omega}/B^{\omega}_{n-1}} \right) .
\end{align}
of the exact sequence \ref{ex2}.  We call this map $\kappa_n$ the Kodaria-Spencer map of order $n$.
\end{theorem}
\begin{remark}
In the case when $n=1$ theorem \ref{main3} comes back to Proposition 9.7 in \cite{Voi02I}. Therefore this theorem can be considered the generalization of
this proposition to any order.
\end{remark}
\section{a new proof of the formula for the obstructions}\label{section4}.

\renewcommand{\theequation}
{4.\arabic{equation}} \setcounter{equation}{-1}

\indent In the section we discuss another application of the cohomology groups defined in section \ref{section2}. In \cite{YEH} and \cite{YET}, two  kinds of obstruction formulas related to the jumping phenomenon of Hodge numbers and cohomology groups of tangent sheaf has been studied. In the following, thses formulas are reformulated in the framework of Doubeault type cohomology groups introduced in section \ref{section2} and new proofs are given by using the analytic method.\\
\indent Recall that,  both of  the operators $\partial_{X_n/B_n}$ and $\bar{\partial}_{X_n/B_n}$ operating on the sections of this sheaf   $\mathcal{A} _{X_n/B_n}^{q}(\Omega _{X_n/B_n}^{p,\omega \infty})=\mathcal{A} _{X_n/B_n}^{p}(\bar{\Omega}_{X_n/B_n}^{q,\omega \infty})
$. Furthermore one gets and an operator ${\partial}_{X_n/B_n}:  H_{\bar{\partial}_{X_n/B_n}}^{q}\left( \Omega^{p}_{\sX/B |X_{n}^{\omega}} \right) \longrightarrow H_{\bar{\partial}_{X_n/B_n}}^{q+1}\left( \Omega^{p}_{\sX/B |X_{n}^{\omega}} \right)$ induced by  ${\partial}_{X_n/B_n}: \mathcal{A} _{X_n/B_n}^{q}(\Omega _{X_n/B_n}^{p,\omega \infty}) \longrightarrow \mathcal{A} _{X_n/B_n}^{q}(\Omega _{X_n/B_n}^{p+1,\omega \infty})$ based on the fact that the composition of $\partial_{X_n/B_n}$ and $\bar{\partial}_{X_n/B_n}$ is commutative. Then theorem 3.5  in \cite{YEH} can be reformulated in the following way.

\begin{theorem} \label{omain}
Let $\pi:\mathcal{X}\longrightarrow B$ be a deformation of
$\pi^{-1}(0)=X$, where $X$ is a compact complex manifold. Let
$\pi_{n}:X_{n}\longrightarrow B_{n}$ be the $n$th order deformation of
$X$. For arbitrary $[\alpha]$ belongs to $H^q(X,\Omega^p)$,
suppose we can extend $[\alpha]$ to order $n-1$ in
$H_{\bar{\partial}_{X_n/B_n}}^{q}\left( \Omega^{p}_{\sX/B |X_{n}^{\omega}} \right)$. Denote such element by
$[\alpha_{n-1}]$. The obstruction of the extension of $[\alpha]$
to $n$th order is given by:
$$ o_{n,n-1}(\alpha)=\partial_{X_{n-1}/B_{n-1}} \circ \kappa_{n} \llcorner (\alpha_{n-1})-
\kappa_{n} \llcorner \circ \partial_{X_{n-1}/B_{n-1}}(\alpha_{n-1}),$$
where $\kappa_{n}$ is the $n$th order Kodaira-Spencer map and
$\partial_{X_{n-1}/B_{n-1}}$ is the relative differential operator of the
$n-1$th order deformation.
\end{theorem}
\begin{remark}
Note that the symbol $\kappa_{n} \llcorner$ has different meaning from the one on  theorem 3.5  in \cite{YEH}, while the latter operates on the $\check{C}$ech cohomology class and the former operates on Dolbeault type cohomology class. It is precisely this difference that causes us to change the plus sign to the minus sign in the formula.
\end{remark}
\begin{proof}
Let $\tilde{\alpha}$ be a section of $\Gamma \left( \mathcal{X} ,\mathcal{A} _{\mathcal{X} /B}^{q}\left( \Omega _{\mathcal{X} /B}^{p} \right) \right)
$ such that
$r_{n-1}\left( \tilde{\alpha} \right) =\alpha _{n-1}$. For any section
$u\in \mathcal{T} _{B_{n}^{\omega}|B_{n-1}^{\omega}}$ , let $\tilde{u}$ is a section on $\mathcal{T} _{B^{\omega}}$
such that $r_{n-1}\tilde{u}=u$.  The obstruction $o_{n,n-1}(\alpha)$ in the direct $u$ is given by $[r_{n-1}(\rho d_{\tilde{u}} \rho^{-1}\bar{\partial}_{\mathcal{X}/B}\tilde{\alpha})]. $
By theorem 2.9 of \cite{XW}
\begin{align*}
[r_{n-1}(\rho d_{\tilde{u}}\rho ^{-1}\bar{\partial}_{\mathcal{X} /B}\tilde{\alpha})]&=[r_{n-1}(\rho d_{\tilde{u}}\rho ^{-1}\bar{\partial}_{\mathcal{X} /B}\rho \rho ^{-1}\tilde{\alpha})]&
\\
&=[r_{n-1}(\rho d_{\tilde{u}}\left( \bar{\partial}-\mathcal{L} _{\varphi \left( t \right)}^{1,0} \right) \rho ^{-1}\tilde{\alpha})]&
\\
&=[r_{n-1}(\rho d_{\tilde{u}}\left( \bar{\partial}-\varphi \left( t \right) \llcorner \partial +\partial \llcorner \varphi \left( t \right) \right) \rho ^{-1}\tilde{\alpha})]&
\\
&=\left[ r_{n-1}\left( \rho \left( -d_{\tilde{u}}\varphi \left( t \right) \llcorner \partial +\partial d_{\tilde{u}}\varphi \left( t \right) \llcorner \right) \rho ^{-1}\tilde{\alpha} \right. \right.&
\\
&\,\, +\left. \left. \rho \left( \bar{\partial}-\varphi \left( t \right) \llcorner \partial +\partial \llcorner \varphi \left( t \right) \right) d_{\tilde{u}}\rho ^{-1}\tilde{\alpha} \right) \right]&
\\
&=\left[ r_{n-1}\left( -\rho d_{\tilde{u}}\varphi \left( t \right) \llcorner \rho ^{-1}\rho \partial \rho ^{-1}\tilde{\alpha}+\rho \partial \rho ^{-1}\rho d_{\tilde{u}}\varphi \left( t \right) \llcorner \rho ^{-1}\tilde{\alpha} \right. \right.&
\\
&\,\, +\left. \left. \rho \left( \bar{\partial}-\varphi \left( t \right) \llcorner \partial +\partial \llcorner \varphi \left( t \right) \right) \rho ^{-1}\rho d_{\tilde{u}}\rho ^{-1}\tilde{\alpha} \right) \right]&
\\
&=\left[ r_{n-1}\left( \left( -\rho d_{\tilde{u}}\varphi \left( t \right) \right) \llcorner \left( \rho \partial \rho ^{-1}\tilde{\alpha} \right) +\rho \partial \rho ^{-1}\left( \rho d_{\tilde{u}}\varphi \left( t \right) \right) \llcorner \tilde{\alpha} \right. \right.
&\\
&\,\, +\left. \left. \bar{\partial}_{\mathcal{X} /B}\rho d_{\tilde{u}}\rho ^{-1}\tilde{\alpha} \right) \right]&
\\
&=\left[ r_{n-1}\left( \left( -\rho d_{\tilde{u}}\varphi \left( t \right) \right) \llcorner \left( \rho \partial \rho ^{-1}\tilde{\alpha} \right) +\rho \partial \rho ^{-1}\left( \rho d_{\tilde{u}}\varphi \left( t \right) \right) \llcorner \tilde{\alpha} \right) \right]&
\\
&+\left[ r_{n-1}\left( \bar{\partial}_{\mathcal{X} /B}\rho d_{\tilde{u}}\rho ^{-1}\tilde{\alpha} \right) \right]&\\
&=\left[ r_{n-1}\left( \left( -\rho d_{\tilde{u}}\varphi \left( t \right) \right) \right)\llcorner  \left( \partial _{\mathcal{X} /B}\tilde{\alpha} \right) +r_{n-1}\left( \partial _{\mathcal{X} /B}\left( \rho d_{\tilde{u}}\varphi \left( t \right) \right) \llcorner  \tilde{\alpha} \right) \right]
&\\
&+\left[ \bar{\partial}_{X_n/B_n}r_{n-1}\left( \rho d_{\tilde{u}}\rho ^{-1}\tilde{\alpha} \right) \right]
&\\
&=\partial_{X_{n-1}/B_{n-1}} \circ \kappa_{n}(u) \llcorner (\alpha_{n-1})-
\kappa_{n}(u) \llcorner \circ \partial_{X_{n-1}/B_{n-1}}(\alpha_{n-1})&
\end{align*}
Thus the formula in the theorem has been proved.

\end{proof}
In the case of tangent sheaf,  theorem 3.4  in \cite{YET} can be reformulated in the following way.

\begin{theorem} \label{omain2}
 Let $\pi:\mathcal{X}\longrightarrow B$ be a deformation of
$\pi^{-1}(0)=X$, where $X$ is a compact complex manifold. Let
$\pi_{n}:X_{n}\longrightarrow B_{n}$ be the $n$th order deformation of
$X$. For arbitrary $[\alpha]$ belongs to $H^q(X,\mathcal{T}_{X})$,
suppose we can extend $[\alpha]$ to order $n-1$ in
$H_{\bar{\partial}_{X_n/B_n}}^{p}\left( \mathcal{T} _{\mathcal{X} /B|X_{n}^{\omega}} \right) .$ Denote such element by
$[\alpha_{n-1}]$. The obstruction of the extension of
$[\alpha]$ to $n$th order is given by:
$$ o_{n,n-1}(\alpha)=[\kappa_{n}, \alpha_{n-1}]_{rel,n-1},$$
where $\kappa_{n}$ is the $n$th order Kodaira-Spencer class and
$[\cdot, \cdot]_{rel,n-1}$ is the Lie bracket induced from the relative tangent sheaf of the
$n-1$th order deformation.
\end{theorem}
\begin{proof}
The proof is again given by direct calculation. Let $\tilde{\alpha}$ be a section of $\Gamma \left( \mathcal{X} ,\mathcal{A} _{\mathcal{X} /B}^{q}\left( \Omega _{\mathcal{X} /B}^{p} \right) \right)
$ such that
$r_{n-1}\left( \tilde{\alpha} \right) =\alpha _{n-1}$. For any section
$u\in \mathcal{T} _{B_{n}^{\omega}|B_{n-1}^{\omega}}$ , let $\tilde{u}$ is a section on $\mathcal{T} _{B^{\omega}}$
such that $r_{n-1}\tilde{u}=u$.  The obstruction $o_{n,n-1}(\alpha)$ in the direct $u$ is given by $[r_{n-1}(\rho d_{\tilde{u}} \rho^{-1}\bar{\partial}_{\mathcal{X}/B}\tilde{\alpha})]. $
By theorem 2.9 of \cite{XW}
\begin{align*}
[r_{n-1}(\rho d_{\tilde{u}}\rho ^{-1}\bar{\partial}_{\mathcal{X} /B}\tilde{\alpha})]&=[r_{n-1}(\rho d_{\tilde{u}}\rho ^{-1}\bar{\partial}_{\mathcal{X} /B}\rho \rho ^{-1}\tilde{\alpha})]&
\\
&=[r_{n-1}(\rho d_{\tilde{u}}\left( \bar{\partial}-\mathcal{L} _{\varphi \left( t \right)}^{1,0} \right) \rho ^{-1}\tilde{\alpha})]
&\\
&=[r_{n-1}(\rho d_{\tilde{u}}\left( \bar{\partial}-\left[ \varphi \left( t \right) ,\cdot \right] \right) \rho ^{-1}\tilde{\alpha})]
&\\
&=[r_{n-1}(\rho \left[ d_{\tilde{u}}\varphi \left( t \right) ,\cdot \right] \rho ^{-1}\tilde{\alpha}+\rho \left( \bar{\partial}-\left[ \varphi \left( t \right) ,\cdot \right] \right) d_{\tilde{u}}\rho ^{-1}\tilde{\alpha})]
&\\
&=[r_{n-1}(\left[ \rho d_{\tilde{u}}\varphi \left( t \right) ,\cdot \right] \tilde{\alpha}+\rho \left( \bar{\partial}-\left[ \varphi \left( t \right) ,\cdot \right] \right) \rho ^{-1}\rho d_{\tilde{u}}\rho ^{-1}\tilde{\alpha})]
&\\
&=[r_{n-1}\left( \left[ \rho d_{\tilde{u}}\varphi \left( t \right) ,\tilde{\alpha} \right] \right) ]+\left[ r_{n-1}(\bar{\partial}_{\mathcal{X} /B}\rho d_{\tilde{u}}\rho ^{-1}\tilde{\alpha}) \right]
&\\
&=[\kappa _n\left( u \right) ,\alpha _{n-1}]_{rel,n-1}.&
\end{align*}
\end{proof}


\bibliographystyle{amsplain}
\bibliography{paper_Ye}

\end{document}